\documentclass[11pt]{article}
\usepackage{a4wide}
\usepackage{amsmath,amssymb,amsthm}
\newcommand{\R}{\mathbb{R}}

\providecommand{\e}{\mathrm{e}}

\providecommand{\e}{\mathrm{e}}
\providecommand{\abs}[1]{\lvert #1 \rvert}
\providecommand{\norm}[1]{\lVert #1 \rVert}
\providecommand{\bary}{\mathop\mathrm{bar}\nolimits}
\providecommand{\tq}{\,\vert\,}
\providecommand{\conv}{\mathop\mathrm{conv}\nolimits}

\providecommand{\vol}{\mathop\mathrm{vol}\nolimits}
\providecommand{\sca}[1]{\langle #1 \rangle}

\newtheorem{theorem}{Theorem}
\newtheorem{corollary}[theorem]{Corollary}
\newtheorem{lemma}[theorem]{Lemma}
\theoremstyle{definition}

\title{A direct proof of the functional Santal\'o inequality}
\author{Joseph Lehec
\footnote{
LAMA (UMR CNRS 8050)
Universit\'e Paris-Est.
}}
\date{June 2008}

\begin{document}
\maketitle
\begin{abstract}
We give a simple proof of a functional version of the Blaschke-Santal\'o inequality due to Artstein, Klartag and Milman. The proof is by induction on the dimension and does not use the Blaschke-Santal\'o inequality.
\bigskip

\noindent
Published in C. R. Acad. Sci. Paris, S\'er.~I  347 (2009) 55--58.
\end{abstract}

\section{Introduction}
For $x,y\in \R^n$, we denote their inner product by $\sca{x,y}$ and the Euclidean norm of $x$ by $\abs{x}$. If $A$ is a subset of $\R^n$, we let $A^\circ= \{ x\in \R^n \tq \forall y\in A ,\, \sca{x,y} \leq 1\}$ be its polar body. The Blaschke-Santal\'o inequality states that any convex body $K$ in $\R^n$ with center of mass at $0$ satisfies
\begin{equation}
\label{classique-santalo}
 \vol_n (K)  \vol_n (K^\circ) \leq  \vol_n(D) \vol_n (D^\circ) = v_n^2 , 
\end{equation}
where $\vol_n$ stands for the volume, $D$ for the Euclidean ball and $v_n$ for its volume. Let $g$ be a non-negative Borel function on $\R^n$ satisfying $0<\int g<\infty$ and $\int \abs{x} g(x) \, dx  < \infty $, then
$\bary(g) = \bigl( \int g \bigr)^{-1} \bigl( \int g(x) x  \, dx \bigr)$
denotes its center of mass (or barycenter). The center of mass (or centroid) of a measurable subset of $\R^n$ is by definition the barycenter of its indicator function.

Let us state a functional form of \eqref{classique-santalo} due to Artstein, Klartag and Milman~\cite{artstein-klartag-milman}. If $f$ is a non-negative Borel function on $\R^n$, the polar function of $f$ is the log-concave function defined by 
\[ f^\circ (x) = \inf_{y\in \R^n} \bigl( \e^{-\sca{x,y} } f(y)^{-1} \bigr) \]
\begin{theorem}[Artstein, Klartag, Milman]
\label{AKM}
If $f$ is a non-negative integrable function on $\R^n$ such that $f^\circ$ has its barycenter at $0$, then
\[
\int_{\R^n}f(x) \, dx   \int_{\R^n}f^\circ(y) \, dy   \leq 
        \bigl( \int_{\R^n} \e^{-\frac{1}{2}\abs{x}^2 } \, dx \bigr)^2
   = (2 \pi)^n .
\]
\end{theorem}
In the special case where the function $f$ is even, this result follows from an earlier inequality of Keith Ball~\cite{these-ball}; and in \cite{fradelizi-meyer}, Fradelizi and Meyer prove something more general (see also~\cite{mon-fm}).
In the present note we prove the following:
\begin{theorem}
\label{monthm}
Let $f$ and $g$ be non-negative Borel functions on $\R^n$ satisfying the duality relation
\begin{equation}
\label{duality}
\forall x,y\in\R^n, \qquad f(x)g(y)\leq \e^{-\sca{x,y}} .
\end{equation}
If $f$ (or $g$) has its barycenter at $0$ then 
\begin{equation}
\label{main-ineq}
  \int_{\R^n}f (x) \, dx 
 \int_{\R^n}g(y) \, dy    \leq   (2\pi)^n .
\end{equation}
\end{theorem}
This is slightly stronger than Theorem~\ref{AKM} in which the function that has its barycenter at $0$ should be log-concave. The point of this note is not really this improvement, but rather to present a simple proof of Theorem~\ref{AKM}.
Theorem~\ref{monthm} yields an improved Blaschke-Santal\'o inequality, obtained by Lutwak in~\cite{lutwak}, with a completely different approach.
\begin{corollary}
Let $S$ be a star-shaped (with respect to $0$) body in $\R^n$ having its centroid at $0$. Then
\begin{equation}
\label{star}
\vol_n (S) \vol_n (S^\circ) \leq v_n^2 .
\end{equation}
\end{corollary}
\begin{proof}
Let $N_S (x) = \inf \{ r>0 \tq x\in r S \}$ be the gauge of $S$ and $\phi_S = \exp \bigl( -\frac{1}{2} N^2_S \bigr)$. Integrating $\phi_S$ and the indicator function of $S$ on level sets of $N_S$, it is easy to see that $\int_{\R^n} \phi_S = c_n \vol_n(S)$ for some constant $c_n$ depending only on the dimension. Replacing $S$ by the Euclidean ball in this equality yields $c_n = (2\pi)^{n/2} v_n^{-1}$. Therefore it is enough to prove that 
\begin{equation}
\label{phiS}
 \int \phi_S \int \phi_{S^\circ} \leq (2\pi)^n .
\end{equation}
Similarly, it is easy to see that $\bary(\phi_S)=c'_n \bary(S)=0$. Besides, we have $\sca{x,y} \leq N_S(x) N_{S^\circ} (y) \leq \tfrac{1}{2} N_S(x)^2 + \tfrac{1}{2} N_{S^\circ}(y)^2$, for all $x,y\in\R^n$. Thus $\phi_S$ and $\phi_{S^\circ}$ satisfy \eqref{duality}, then by Theorem~\ref{monthm} we get \eqref{phiS}.
\end{proof}
\section{Main results}
\begin{theorem}
\label{main}
Let $f$ be a non-negative Borel function on $\R^n$ having a barycenter. Let $H$ be an affine hyperplane splitting $\R^n$ into two half-spaces $H_+$ and $H_-$. Define $\lambda\in [0,1]$ by $\lambda \int_{\R^n} f =\int_{H_+} f$. Then there exists $z\in \R^n$ such that for every non-negative Borel function $g$
\begin{equation}
 \label{lambda}
\bigl( \forall x,y \in\R^n, \; f(z+x)g(y) \leq \e^{-\sca{x,y}} \bigr) \quad \Rightarrow \quad
\int_{\R^n} f \int_{\R^n} g  \leq \frac{1}{4\lambda(1-\lambda)} (2\pi)^n.
\end{equation}
In particular, in every median $H$ ($\lambda=\frac{1}{2}$) there is a point $z$ such that for all $g$
\begin{equation}
\label{median} 
\bigl( \forall x,y \in\R^n, \; f(z+x)g(y) \leq \e^{-\sca{x,y}} \bigr) \quad \Rightarrow \quad
\int_{\R^n} f \int_{\R^n} g  \leq (2\pi)^n.
\end{equation}
\end{theorem}
A similar result concerning convex bodies (instead of functions) was obtained by Meyer and Pajor in \cite{meyer-pajor2}. 

Let us derive Theorem~\ref{monthm} from the latter.
Let $f,g$ satisfy \eqref{duality}. Assume for example that $\bary(g) = 0$, then $0$ cannot be separated from the support of $g$ by a hyperplane, so there exists $x_1,\dots, x_{n+1} \in \R^n$ such that $0$ belongs to the interior of $\conv\{x_1 \dots x_{n+1}\}$ and $g(x_i) >0$ for $i=1\dots n+1$. Then \eqref{duality} implies that $f(x) \leq C \e^{- \norm{x} }$, for some $C>0$, where  $\norm{x} = \max \bigl( \sca{x,x_i} \tq i\leq n+1 \bigr)$. Assume also that $\int f >0$, then $f$ has a barycenter. Apply the ``$\lambda=1/2$'' part of Theorem~\ref{main} to $f$. There exists $z\in\R^n$ such that \eqref{median} holds. On the other hand, by \eqref{duality}
\[ f(z+x) g(y) \e^{\sca{y,z} } \leq \e^{- \sca{z+x,y} } \e^{ \sca{y,z} } = \e^{- \sca{x,y} } \]
for all $x,y \in \R^n$. Therefore 
\begin{equation}
 \label{etapez}
\int_{\R^n} f(x) \, dx \int_{\R^n} g(y) \e^{\sca{y,z}} \, dy \leq (2\pi)^n .
\end{equation}
Integrating with respect to $g(y) dy$ the inequality $1 \leq \e^{\sca{y,z}}- \sca{y,z}$ we get
\[ \int_{\R^n} g(y) \, dy  \leq
\int_{\R^n} g(y)  \e^{\sca{y,z}}  \, dy - \int_{\R^n} \sca{y,z} g(y) \, dy . \]
Since $\bary(g)=0$, the latter integral is $0$ and together with \eqref{etapez} we obtain \eqref{main-ineq}. Observe also that this proof shows that Theorem~\ref{main} in dimension $n$ implies Theorem~\ref{monthm} in dimension $n$.

In order to prove Theorem~\ref{main}, we need the following logarithmic form of the Pr\'ekopa-Leindler inequality. For details on Pr\'ekopa-Leindler, we refer to \cite{ball}.
\begin{lemma}
\label{geometric-mean}
Let $\phi_1,\phi_2$ be non-negative Borel functions on $\R_+$. If $\phi_1(s) \phi_2(t) \leq \e^{-st}$
for every $s,t$ in $\R_+$, then 
\begin{equation}
\label{gm}
\int_{\R_+}  \phi_1(s) \, ds \int_{\R_+}  \phi_2(t) \, dt \leq \frac{\pi}{2} .
\end{equation}
\end{lemma}
\begin{proof}
Let $f(s) = \phi_1 (\e^s) \e^s$, $g (t) = \phi_2 (\e^t) \e^t$ and $h(r) = \exp (-\e^{2r}/2) \e^r$. For all $s,t\in\R$ we have $\sqrt{f(s)g(t)}\leq h(\frac{t+s}{2})$, hence by Pr\'ekopa-Leindler $\int_\R f \int_\R g \leq \bigl( \int_\R h \bigr)^2$. By change of variable, this is the same as $\int_{\R_+}  \phi_1 \int_{\R_+}  \phi_2 \leq 
 \bigl( \int_{\R_+} \e^{-u^2/2} \, du  \bigl)^{2}$ which is the result.
\end{proof}
\section{Proof of Theorem~\ref{main}}
Clearly we can assume that $\int f =1$. Let $\mu$ be the measure with density $f$. In the sequel we let $f_z(x) = f(z+x)$ for all $x,z$.

We prove the theorem by induction on the dimension. Let $f$ be a non-negative Borel function on the line, let $r\in \R$ and $\lambda = \mu\bigl([r,\infty)\bigr) \in[0,1]$. Let $g$ satisfy $f(r+s) g(t) \leq \e^{-st}$, for all $s,t$. Apply Lemma~\ref{geometric-mean} twice: first to $\phi_1(s)=f(r+s)$ and $\phi_2(t)= g(t)$ then to $\phi_1(s)=f(r-s)$ and $\phi_2(t)= g(-t)$. Then
\[ \int_{\R_+} f_r   \int_{\R_+} g   \leq \frac{\pi}{2}  \qquad \text{and} \qquad 
 \int_{\R_-} f_r   \int_{\R_-} g \leq \frac{\pi}{2} . \]
Therefore $\int_{\R_+} g   \leq \frac{\pi}{2\lambda}$ and $\int_{\R_-} g   \leq \frac{\pi}{2 (1-\lambda)}$, which yields the result in dimension $1$. \\
Assume the theorem to be true in dimension $n-1$. Let $H$ be an affine hyperplane splitting $\R^n$ into two half-spaces $H_+$ and $H_-$ and let $\lambda=\mu(H_+)$. Provided that $\lambda\neq 0,1$ we can define $b_+$ and $b_-$ to be the barycenters of $\mu_{|H_+}$ and $\mu_{|H_-}$, respectively. Since $\mu (H) = 0$, the point $b_ +$ belongs to the interior of $H_+$, and similarly for $b_-$. Hence the line passing through $b_{+}$ and $b_{-}$ intersects $H$ at one point, which we call $z$. Let us prove that $z$ satisfies \eqref{lambda}, for all $g$. Clearly, replacing $f$ by $f_z$ and $H$ by $H-z$, we can assume that $z=0$. Let $g$ satisfy 
\begin{equation}
 \label{reduality}
\forall x,y\in \R^n, \qquad f(x) g(y) \leq \e^{- \sca{x,y}} . 
\end{equation}
Let $e_1,\dots,e_n$ be an orthonormal basis of $\R^n$ such that $H = e_n^\perp$ and $\sca{b_+,e_n}>0$. Let
$v = b_+ / \sca{b_+ , e_n}$ and $A$ be the linear operator on $\R^n$ that maps $e_n$ to $v$ and $e_i$ to itself for $i=1\dots n-1$ and let $B=( A^{-1} )^t$. Define
\[ F_+: y\in H \mapsto \int_{\R_+} f( y + s v) \, ds  \qquad \text{and} \qquad
G_+: y'\in H  \mapsto \int_{\R_+} g (B y' +te_n) \, dt . \] 
By Fubini, and since $A$ has determinant $1$, $\int_H F_+  = \int_{H_+} f\circ A=\mu(H_+)=\lambda$. Also, letting $P$ be the projection with range $H$ and kernel $\R v$, we have
\[  \bary(F_+) = \frac{1}{\lambda} \int_{H_+} P(Ax) f(Ax) \, dx = \frac{1}{\lambda} P \Bigl( \int_{H_+} x f(x) \, dx \Bigr) = P ( b_+ ) , \] 
and this is $0$ by definition of $P$. Since $\sca{Ax,Bx'}= \sca{x,x'}$ for all $x,x' \in \R^n$, we have $ \sca{ y+sv  , By'+te_n } = \sca{y,y'} +st$ for all $s,t\in \R$ and $y,y'\in H$. So \eqref{reduality} implies
\[  f ( y + sv ) g ( By'+te_n)  \leq \e^{-st- \sca{y,y'} } . \]
Applying Lemma~\ref{geometric-mean} to $\phi_1(s)= f(y+sv)$ and $\phi_2(t)=g(By'+te_n)$ we get $F_+(y)  G_+ (y')  \leq  \frac{\pi}{2}  \e^{-\sca{y,y'}}$ for every $y,y'\in H$. Recall that $\bary(F_+)=0$, then by the induction assumption (which implies Theorem~\ref{monthm} in dimension $n-1$)
\begin{equation}
\label{etape3}
 \int_H F_+  \int_H G_+  \leq  \frac{\pi}{2} (2\pi)^{n-1} .
\end{equation}
hence $\int_{H_+} g (Bx) \, dx  \leq  \frac{1}{4\lambda} (2\pi)^n$. In the same way $ \int_{H_-} g (Bx) \, dx  \leq  \frac{1}{4(1-\lambda)} (2\pi)^n $, adding these two inequalities, we obtain 
\[ \int_{\R^n} g(Bx) \, dx \leq \frac{1}{4\lambda(1-\lambda)} (2\pi)^n \]
which is the result since $B$ has determinant $1$.


\begin{thebibliography}{00}

\bibitem{artstein-klartag-milman}
{S.~Artstein-Avidan, B.~Klartag, and V.~Milman}, \emph{The Santal\'o point of a function, and a functional form of {S}antal\'o inequality}, Mathematika {\bf 51} (2005) 33--48.
%
\bibitem{these-ball}
{K.~Ball}, \emph{Isometric problems in $\ell_p$ and sections of convex sets}, doctoral thesis, University of Cambridge, 1986.
%
\bibitem{ball}
{K.~Ball}, An elementary introduction to modern convex geometry, in {\em Flavors
  of geometry}, edited by S.~Levy, Cambridge University Press, 1997.
%
\bibitem{fradelizi-meyer}
M.~Fradelizi and M.~Meyer,
\emph{Some functional forms of {B}laschke-{S}antal\'o inequality}, {Math. Z.} {\bf 256} (2007) (2) 379--395.
%
\bibitem{mon-fm}
J.~Lehec, \emph{Partitions and functional Santal\'o inequalities}, Arch. Math. 92 (2009) (1) 89--94.
%
\bibitem{lutwak}
E.~Lutwak, \emph{Extended affine surface area}, Adv. Math. {\bf 85} (1991) (1) 39--68.
%
\bibitem{meyer-pajor2}
{M.~Meyer and A.~Pajor}, \emph{On the Blaschke Santal\'o inequality}, {{A}rch.
  {M}ath. (Basel)} {\bf 55} (1990) 82--93.
%
\end{thebibliography}
\end{document}